\documentclass[fleqn,11pt]{amsart}            

\usepackage{verbatim}
\usepackage{graphicx}
\usepackage{graphics}
\usepackage{latexsym}
\usepackage{amssymb,amsmath, MnSymbol}
\usepackage{multicol}
\usepackage{bussproofs}
\usepackage{enumerate}
\usepackage{xcolor}
\usepackage[latin1]{inputenc}

\date{}

                                  \usepackage{amsthm,latexsym}

\numberwithin{equation}{section}

\newtheorem{definicion}{Definition}[section]

\newtheorem{definition}[definicion]{Definition}
\newtheorem{Lemma}[definicion]{Lemma}

\newtheorem{Theorem}[definicion]{Theorem}
\newtheorem{theorem}[definicion]{Theorem}
\newtheorem{Corollary}[definicion]{Corollary}

\newtheorem{remark}[definicion]{Remark}

\newenvironment{Proof}{\noindent\bf Proof \rm}{$\hfill
\square$}

\newcounter{ecuacionDef} 
\renewcommand{\theecuacionDef}{\arabic{ecuacionDef}}

{
	\refstepcounter{ecuacionDef}
	\vspace{0.12cm}
	{\rm (E\theecuacionDef)}:  #1
	\vspace{0.12cm}
}%

\newcounter{numeroAxioma} 
\renewcommand{\thenumeroAxioma}{\arabic{numeroAxioma}}

{
	\refstepcounter{numeroAxioma}
	\vspace{0.12cm}
	(A\thenumeroAxioma):  #1
	\vspace{0.12cm}
}%

\begin{document}

 \newcounter{thlistctr}
 \newenvironment{thlist}{\
 \begin{list}%
 {\alph{thlistctr}}%
 {\setlength{\labelwidth}{2ex}%
 \setlength{\labelsep}{1ex}%
 \setlength{\leftmargin}{6ex}%
 \renewcommand{\makelabel}[1]{\makebox[\labelwidth][r]{\rm (##1)}}%
 \usecounter{thlistctr}}}%
 {\end{list}}

\title[Amalgamation Property in Almost Gautama Algebras]{Amalgamation Property 
in the subvarieties of \\
Gautama and Almost Gautama Algebras} 

\author[J. M. CORNEJO \and 
H. P. SANKAPPANAVAR] {Juan M. CORNEJO* \and 
Hanamantagouda P. SANKAPPANAVAR**}

\newcommand{\acr}{\newline\indent}

\address{\llap{*\,} Departamento de Matem\'atica\acr
Universidad Nacional del Sur\acr
Alem 1253, Bah\'ia Blanca, Argentina\acr
INMABB - CONICET}

\email{jmcornejo@uns.edu.ar}

\address{\llap{**\,}Department of Mathematics\acr
                              State University of New York\acr
                              New Paltz, New York, 12561\acr
                              U.S.A.}

\email{sankapph@hawkmail.newpaltz.edu}


\maketitle

\begin{abstract}
Gautama algebras were introduced recently, 
as a common generalization of regular double Stone algebras and regular Kleene Stone algebras.  Even more recently,  Gautama algebras were further generalized to Almost Gautama algebras ($\mathbb{AG}$ for short).
The main purpose of this paper 
is to investigate the Amalgamation Property (AP, for short) in the subvarieties of the variety  $\mathbb{AG}$.   In fact, we show that, of the eight nontrivial subvarieties of  $\mathbb{AG}$,  
only four varieties, namely those of Boolean algebras, of regular double Stone algebras, of regular Kleene Stone algebras and of De Morgan Boolean algebras have the AP and the remaining four do not have the AP.  We give several applications of this result; in particular,
we examine the following properties for the subvarieties of $\mathbb{AG}$: transferability property (TP), having enough injectives (EI), Embedding Property, Bounded Obstruction Property and having a model companion.
\end{abstract}

\medskip 
\section{ INTRODUCTION}

Around the middle of the nineteenth century, Boole introduced 
(see \cite{Bo47, Bo54}) Boole's partial algebras in an attempt to cast the classical propositional logic into an algebraic setting.  About ten years later, Jevons \cite{Je64} gave a modification of those partial algebras, by replacing Boole's partial operations with total operations.   Around the beginning of the twentieth century, those modified algebras came to be known as ``Boolean algebras.''   Stone algebras--a generalization of Boolean algebras-- were initiated  in 1950's by Gr\"atzer and Schmidt (see \cite{GrSc57}) in an attempt to answer a problem of Stone.  Double Stone algebras arose in 1970's by considering an expansion of a Stone algebra obtained by adding the dual Stone operation.   
Soon thereafter, regular double Stone algebras were introduced by Varlet in 1973 to characterize those double Stone algebras which are congruence-regular, by means of a regularity condition which was a quasi-equation.  Katri\v{n}\'{a}k \cite{Ka73} studied this quasi-equation and showed that it is equivalent to an identity (R), thus showing that the class of regular double Stone algebras is a subvariety of the variety of double Stone algebras.  By suitably modifying Katri\v{n}\'{a}k's identity, Sankappanavar \cite{Sa86} introduced a notion of regularity in the variety of Kleene Stone algebras--actually in a much larger variety-- from which it was clear that the regular   
Kleene Stone algebras form a variety. 

The definitions of regular double Stone algebras and regular Kleene Stone algebras 
 and several results about these two classes are quite similar (see Section 2).  This similarity has recently given rise to a new variety of algebras called ``Gautama algebras'' (see \cite{Sa22}), 
of which the varieties of regular double Stone algebras and regular Kleene Stone algebras are subvarieties.  Even more recently, in \cite{CoSa23b}, Gautama algebras are further generalized to Almost Gautama algebras ($\mathbb{AG}$ for short) by a slight weakening of one of their defining identities.
 
The main goal of this paper 
is to investigate the Amalgamation Property (AP, for short) in the subvarieties of the variety  $\mathbb{AG}$.   In fact, we show in our main result (Theorem 3.24) that, of the eight nontrivial subvarieties of  $\mathbb{AG}$,  
only four varieties, namely those of Boolean algebras, of regular double Stone algebras, of regular Kleene Stone algebras and of De Morgan Boolean algebras have the AP and the remaining four do not have the AP.  We give several applications of this result; More specifically,
we examine such properties for the subvarieties of $\mathbb{AG}$ as transferability property (TP), having enough injectives (EI), Embedding Property, Bounded Obstruction Property and having a model companion.

 The paper is organized as follows: In Section 2, we present some preliminary concepts and known results that will be needed in the rest of the paper.  Section 3, the heart of this paper, gives a complete description of the subvarieties of the variety of Almost Stone algebras having the amalgamation property.  The main theorem (Theorem 3.24) 
describes precisely which of the eight nontrivial subvarieties of the variety of Almost Gautama algebras have the Amalgamation Property.   
 In section 4, we give applications of our main theorem, combined with some earlier known general results.  In fact, we characterize the subvarieties of $\mathbb{AG}$ having the properties: (TP), (EI), Embedding property, Bounded Obstruction Property and the property of having a model companion.

\smallskip

\section{Preliminaries}

It is assumed here that the reader has had some familiarity with lattice theory and universal algebra (see \cite{BaDw74, BuSa81}, for example).  
As such, for notions, notations and results assumed here, the reader can refer to these or other relevant books.  It will be helpful, but not necessary, to have some familiarity with \cite{Sa22, CoSa23b}.

Recall that an algebra $\mathbf A =\langle A, \lor,\land, ^c, 0 ,1\rangle$ is a Boolean algebra if $\mathbf A$ is a complemented distributive lattice.   
 Let $\mathbb{BA}$ denote the variety of Boolean algebras.  It is well-known that $\mathbb{BA}= \mathbb{V(\mathbf{2})}$ (i.e. the variety generated by $\{\mathbf{2}\})$, where $\mathbf{2}$ denotes the 2-element Boolean algebra with the universe $\{0,1\}$.    
 In what follows, 
we will be casual in that any 2-element Boolean algebra in other languages will also be denoted by $\mathbf{2}.$

Among the many generalizations of Boolean algebres, we consider here the following three: \\ 
(1) Stone algebras, (2) dual Stone algebras and (3) Kleene algebras.

 Stone algebras are a subvariety of the variety of pseudocomplemented lattices.  
We present here the following definition
given in 
 \cite[Corollary 2.8]{Sa87c}.

\smallskip
An algebra $\mathbf{A}=\langle A, \lor, \land, ^*, 0, 1\rangle$ is a {\bf distributive pseudocomplemented lattice} ($p$-algebra, for short) if 
$\mathbf{A}$ satisfies the following:

\rm(1) $\langle A, \lor, \land, 0, 1\rangle$ is a bounded distributive lattice,

(2) the operation $^*$ satisfies the identities: 

\quad  (a) $0^* \approx 1$, 
 
\quad   (b) $1^* \approx 0$, 
  
\quad   (c) $(x \lor y)^* \approx x^* \land y^*$, 
  
\quad   (d) $(x \land y)^{**} \approx x^{**} \land y^{**}$, 
  
\quad   (e)  $x \leq x^{**}$,    
  
\quad   (f) $x^* \land x^{**} \approx 0$. 

Note that the identity (f) can be replaced by the identity: $x \land x^* \approx 0$.  

 A $p$-algebra $\mathbf{A}$ is a {\bf Stone algebra} if $\mathbf{A}$ satisfies the identity:
 
\quad   (St) $x^* \lor x^{**} \approx 1$ (Stone identity).

Dual Stone algebras are, of course, defined dually.

\smallskip
The variety of Kleene algebras is a subvariety of that of De Morgan algebras, which were first introduced by Moisil \cite{Mo35}) in 1935  
(see also \cite{BiRa57} and \cite{Ka58}).  

\medskip
An algebra $\langle A, \lor, \land, ', 0, 1\rangle$ is a {\bf De Morgan algebra} if

\quad (1) $\langle A, \lor, \land, 0, 1\rangle$ is a bounded distributive lattice,

\quad (2) $0' \approx 1$ and $1' \approx 0$,

\quad (3) $(x \land y)' \approx x' \lor y'$ ($\land$-De Morgan law),
 
\quad (4) $x'' \approx x$  (Involution).

A De Morgan algebra is a {\bf Kleene algebra} if it satisfies:

\quad (5) $x \land x' \leq y \lor y'$ (Kleene identity).

\medskip

Let a 3-element chain (viewed as a bounded distributive lattice) be denoted by $\mathbf{3}$. 
 The three operations, namely the pseudocomplement, the dual pseudocomplement and the De Morgan operation on the three-element chain are shown below in Figure 1.\\

\vspace{.5cm}
\setlength{\unitlength}{1cm} \vspace*{.5cm}
\begin{picture}(15,2)
\put(1.5,-.4){\circle*{.15}}          

\put(1.5,1){\circle*{.15}}

\put(1.5,2.3){\circle*{.15}}

\put(1.1,-.4){$0$}                      

\put(1.1,.9){$a$}

\put(1.1,2.1){$1$}

\put(.1,.7){${\bf 3}$\ :}

\put(1.5,-.4){\line(0,1){2.7}}

\hspace{2cm}
\begin{tabular}{r|rrr}
$^*$ & 0\ & a & 1\\
\hline
   & 1\ & 0\ & 0\\
\end{tabular} \hspace{.5cm} 
\begin{tabular}{r|rrr}
$^+$ & 0 & a &1 \\
\hline
   & 1 & 1 & 0\\
\end{tabular} \hspace{.5cm} 
\begin{tabular}{r|rrr}
$'$ & 0 & a &1 \\
\hline
   & 1 & a & 0\\
\end{tabular} \hspace{.5cm}
\put(-5,-1.4){Figure 1}  \label{Fig1}  
\end{picture}

\vspace{2cm}

Expansions of the language of bounded lattices by adding
two unary operation symbols corresponding to any two of the above three unary operations give rise to the following three algebras on the 3-element chain: 
 $\mathbf{3_{dblst}} = \langle 3, \lor, \land, ^*, ^+, 0, 1\rangle$, 
 $\mathbf{3_{klst}} = \langle 3,  \lor, \land, ^*, ', 0, 1\rangle$, and
 $\mathbf{3_{kldst}} = \langle 3, \land,  \lor, ^+, ', 0, 1\rangle$.
 The algebra $\mathbf{3_{dblst}}$ 
 is known as a ``regular double Stone algebra'' and 
 $\mathbf{3_{klst}}$ as a ``regular Kleene Stone algebra'', while
 $\mathbf{3_{kldst}}$, 
 being the dual of $\mathbf{3_{klst}}$, would not be of much interest to us in this paper.   

An algebra $\mathbf{A}=\langle A, \lor, \land, ^*, ^+, 0, 1\rangle$ is a {\bf regular double Stone algebra} if

\rm(1)  $\langle A, \lor, \land, ^*, 0, 1\rangle$ is a Stone algebra,

(2)  $\langle A, \lor, \land, ^+, 0, 1\rangle$ is a dual Stone algebra,

(3)  $\mathbf{A}$ satisfies the identity:

\qquad (R)  $x \land x^+ \leq y \lor y^*$;
 or, equivalently, 
 
 \qquad (R1) $x \land x^{+*+}  \leq y \lor y^*$.

The variety of regular double Stone algebras is denoted by $\mathbb{RDBLS}\rm t$.   This variety was introduced in \cite{Ka73} and has later been investigated in many articles, see, for example, \cite{AdSaCa19, CoKiSa23a, Sa11, Sa22}. 

An algebra $\mathbf{A}=\langle A, \lor, \land, ^*, ', 0, 1\rangle$ is a {\bf regular Kleene Stone algebra} if

\rm(1)  $\langle A, \lor, \land, ^*, 0, 1\rangle$ is a Stone algebra,

(2)  $\langle A, \lor, \land, ', 0, 1\rangle$ is a Kleene algebra,

(3)  $\mathbf{A}$ satisfies the identity:

\qquad \rm(R1)  $x \land x'{^*}' \leq y \lor y^*$ \quad (Regularity).

The variety of regular Kleene Stone algebras is denoted by $\mathbb{RKLS}\rm t$.
This variety was introduced in \cite{Sa86} and has later been investigated by several  authors--see, for example, \cite{AdSaCa20, CoSa22, CoKiSa23a,  CoSa24a, CoSa24b, Sa11, Sa22}. 
\medskip

The following two theorems list some of the known properties of the varieties $\mathbb{RDBLS}\rm t$ and $\mathbb{RKLS}\rm t$.
\begin{Theorem}  \cite{Sa11} \label{T1}  
\begin{thlist}
\item
[i]  $\mathbb{RDBLS}\rm t = \mathbf{V(3_{dblst})}$, 

\item
[ii] The variety $\mathbb{RDBLS}\rm t$ is a discriminator variety; and so 
 $\mathbf{3_{dblst}}$ is quasiprimal; and  
$\mathbb{BA}$ is the only nontrivial proper subvariety of $\mathbb{RDBLS}\rm t$. 
\end{thlist}
\end{Theorem} 

\begin{Theorem} \cite{Sa11}  \label{T2}                    
\begin{thlist}  

\item
[i]  $\mathbb{RKLS}\rm t= \mathbb{V}\mathbf{(3_{klst})}$, 

\item
[ii]  $\mathbb{RKLS}\rm t$ is a discriminator variety, 
and $\mathbf{3_{klst}}$ is quasiprimal,  
\end{thlist}
\end{Theorem}

The similarities of $\mathbb{RDBLS}\rm t$ and $\mathbb{RKLS}\rm t$   
led, recently, to the variety of Gautama algebras in \cite{Sa22}, named in honor and memory of Medhatithi Gautama and Aksapada Gautama, the founders of Indian Logic.
To define Gautama algebras, 
we need the dual of the notion of
an ``upper quasi-De Morgan algebra''  which was introduced    
in 1987.  
(We drop the word ``upper.'') 

Quasi-De Morgan algebras arose as a subvariety of semi-De Morgan algebras which were introduced in \cite{Sa87c} to unify pseudocomplemented lattices and De Morgan algebras.  We need the dual of this notion here.
  
An algebra $\mathbf A = \langle A , \lor, \land, ',  0, 1\rangle$ is a {\bf dually quasi-De Morgan algebra} if the following  conditions hold:

(a) $\langle A , \lor, \land, 0, 1\rangle$ is a bounded distributive lattice,

(b) The operation $'$ is a dual quasi-De Morgan operation; that is, $ '$ satisfies\

\quad (i)	 $0' \approx 1$ and $1' \approx 0$,

\quad (ii)	 $(x \land y)' \approx x' \lor y'$,

\quad (iii)	$(x \lor y)'' \approx x'' \lor y''$,

\quad (iv)	$x'' \leq x$.

The variety of dually quasi-De Morgan algebras is denoted by $\mathbb{DQD}$. \\

We can now define ``Gautama algebras.''

\begin{definition} \cite{Sa22}
An algebra $\mathbf{A} = \langle A, \lor, \land, ^*, ' , 0, 1\rangle$ is a {\bf Gautama algebra} if the following conditions hold:

\rm(a) $\langle A, \lor, \land, ^*,  0, 1\rangle$ is a Stone algebra,

(b) $\langle A, \lor, \land, ' , 0, 1\rangle$ is a dually quasi-De Morgan algebra,

(c) $\mathbf{A}$ is regular; i.e., $\mathbf{A}$ satisfies: 

\quad (R1) \quad $x \land x'{^*}' \leq y \lor y^*$,

(d)  $\mathbf{A}$ is star-regular; i.e., $\mathbf{A}$ satisfies the identity: 

\quad  \rm (*)  \quad   $x{^*}' \approx x^{**}$.
\end{definition}
Let $\mathbb{G}$ denote the variety of Gautama algebras.  It should be  remarked here that, in the presence of (b), (c), and (a'): $\langle A, \lor, \land, ^*,  0, 1\rangle$ is a $p$-algebra, it is known \cite{CoSa24c}, but still unpublished, that the star-regular identity implies the Stone identity. 

Clearly, $\mathbf{2}$, $\mathbf{3_{dblst}}$, $\mathbf{3_{klst}}$ are examples of algebras in $\mathbb{G}$. 
 
The following theorem, proved in \cite{Sa22}, gives a concrete description of the subdirectly irreducible algebras in the variety $\mathbb{G}$.

\begin{Theorem} \cite{Sa22}\label{Th1}
Let $\mathbf A \in \mathbb{G}$. Then the following are equivalent:

\rm(1) $\mathbf A$ is simple; 

(2) $\mathbf A$ is subdirectly irreducible; 

(3) $\mathbf A$ is directly indecomposable; 

(4) For every $x \in A$, \ $x \lor x^* =1$ \ implies \  $x=0$  or  $x=1$;           

(5) $\mathbf A \in \{\mathbf 2, \mathbf 3_\mathbf{dblst}, \mathbf 3_\mathbf{klst} \}$, up to isomorphism.  
\end{Theorem}        

Soon after the chapter \cite{Sa22} appeared in print, we noticed that the star-regular identity implies the identity: $x{^*}'' \approx x^*$ and the algebra $\bf 4_{dmba}$ satisfies this weaker identity but not the star-regular identity.  These observations recently gave rise to a
further generalization of Gautama algebras in \cite{CoSa23b}, called
 ``Almost Gautama algebras.'' 

\begin{definition}
An algebra $\mathbf{A} = \langle A, \lor, \land, ^*, ' , 0, 1\rangle$ is an {\bf Almost Gautama algebra} if the following conditions hold:

\rm(a) $\langle A, \lor, \land, ^*,  0, 1\rangle$ is a Stone algebra,  

(b)  $\langle A, \lor, \land, ' , 0, 1\rangle$ 
is a dually quasi-De Morgan algebra,

(c) $\mathbf{A}$ is regular.  That is, $\mathbf{A}$ satisfies: 

\quad (R1) \  $x \land x'{^*}' \leq y \lor y^*$ (Regularity),

(d) $\mathbf{A}$ is Weak Star-Regular.  That is,  $\mathbf{A}$ satisfies the identity: 

 \quad \ ($^*\rm)_w$  \   $x{^*}'' \approx x^*$, (weak star-regularity),

(e) $\mathbf{A}$ satisfies the identity:

\quad  (L1) \  $(x \land x'^*)'^* \approx x \land x'^*$.
\end{definition}
Let $\mathbb{AG}$ denote the variety of Almost Gautama algebras.  It is recently known (see  \cite{CoSa24c}) that in the presence of (b), (c), (e) and (a'): $\langle A, \lor, \land, ^*,  0, 1\rangle$ is a $p$-algebra, the weak star-regular identity and the Stone identity are equivalent; and hence one of the two is redundant in the above definition of Almost Gautama algebra.

Note that the varieties $\mathbb{BA}$ of Boolean algebras, $\mathbb{RDBLS}\rm t$ of regular double Stone algebras, and $\mathbb{RKLS}\rm t$ of regular Kleene Stone algebras and the variety $\mathbb{G}$ of Gautama algebras are all subvarieties of the variety $\mathbb{AG}$ of Almost Gautama algebras.

In Figure 2, we describe a $4$-element 
algebra $\mathbf 4_{\mathbf{dmba}} := \langle \{0, a, b, 1\}, \lor, \land, ^*, ', 0, 1\rangle$, in which $*$ is the Boolean complement with $a^*=b$, $b^*=a$; and $0'=1$, $1'=0$, $a'=a$ and $b'=b.$  It is easy to see that $\mathbf 4_{\mathbf{dmba}}$
is an Almost Gautama algebra.  Observe that $\mathbf 4_{{\bf dmba}}$ is not a Gautama algebra (e.g., take $x:=b$ in $(^*)_w$).\\

\vspace{.4cm}
\setlength{\unitlength}{.7cm} \vspace*{1cm}
\begin{picture}(15,2)
 \put(9.75,3.8){\circle*{.28}}   

  \put(8.6,2.7){\circle*{.28}}  
  \put(10.9, 2.7){\circle*{.28}}  
 \put(9.7,1.5){\circle*{.28}} 

\put(9.9, 1.1){$0$}

\put(8.2,2.3){$b$}
\put(11.1,2.2){$b^*$}
\put(10.0,3.7){$1$}

\put(4.5,1.6){$\bf 4_{dmba}:$}

\put(9.7,1.5){\line(1,1){1.2}} 
\put(9.7,1.5){\line(-1,1){1.2}} 
\put(8.5,2.6){\line(1,1){1.2}}  
\put(10.95,2.6){\line(-1,1){1.2}} 
\put(8.6,0.2){Figure 2}
\end{picture}

The following theorem, which gives an explicit description of subdirectly irreducible Almost Gautama algebras, is proved in \cite{CoSa23b}.

\begin{Theorem} {\rm (}\cite[Theorem 3.12]{CoSa23b}{\rm )} \label{Main1}
Let $\mathbf A \in \mathbb{AG}$. Then the following are equivalent:

{\rm(1)} $\mathbf A$ is simple;

{\rm(2)} $\mathbf A$ is subdirectly irreducible;

{\rm(3)}  $\mathbf A$ is directly indecomposable;

{\rm(4)} $\mathbf A$  satisfies (SC): For every $x \in A$, \ $x \neq 1 \Rightarrow x \land x'^*=0$;

{\rm(5)} $\mathbf A \in \{\mathbf 2, \mathbf{3_{dblst}}, \mathbf{3_{dmst}}, \mathbf{4_{dmba}} \}$,
where $\mathbf{4_{dmba}}$  is the algebra in Figure 2.
\end{Theorem}
The following corollary is also taken from \cite[Corollary 3.33]{CoSa23b}.

\newpage
\begin{Corollary} \label{subvarietiesofAG}
\begin{thlist}
\item[1]
The lattice of nontrivial subvarieties of $\mathbb{AG}$ is isomorphic to the 
eight-element Boolean lattice  
with $\mathbb V(\mathbf 2)$ {\rm(}i.e., the variety of Boolean allgebras{\rm)} as the least element.  The Hasse diagram of this lattice is given in 
{\rm Figure 3}. 
  
 \item[2]
 The lattice of nontrivial subvarieties of $\mathbb{G}$ is isomorphic to the 4-element Boolean lattice with $\mathbb V(\mathbf 2)$
as the least element {\rm(}see {\rm Figure 3 also}{\rm)}.  
 \end{thlist}
\end{Corollary}

\setlength{\unitlength}{.7cm} 

\begin{figure} [ht]
	\begin{center}

		\unitlength 1.9mm 
		\linethickness{0.8pt}
		\ifx\plotpoint\undefined\newsavebox{\plotpoint}\fi 
		\hspace{-5cm} \vspace{-1cm} \begin{picture}(60,57.5)(0,0)   
			
			
			\put(44,43){\circle*{1.5}}   
			\put(47,43.25){$\mathbb{AG}$}
			
			\put(34.25,33.25){\circle*{1.5}} 
			\put(17,33){$\mathbb{V}(\mathbf{3_{dblst}}, \mathbf{4_{dmba}})$} %
			\put(54,33){\circle*{1.5}}      
			\put(56,33){$\mathbb{V}(\mathbf{3_{klst}}, \mathbf{4_{dmba}})$}  
			\put(44,33){\circle*{1.5}}   
			\put(46,33){$\mathbb{G}$}
			
			\put(44,23){\circle*{1.5}}    
			\put(44.5, 22){$\mathbf{V}(\mathbf{4}_{dmba})$} 
			\put(54,23){\circle*{1.5}}  
			\put(56,22){$\mathbb{V}(\mathbf{3_{klst}})$}
			\put(34,23){\circle*{1.5}}   
			\put(23,22){$\mathbb{V}(\mathbf{3_{dblst}})$}

			\put(44,13){\circle*{1.5}}   
			\put(46,10){$\mathbb V(\mathbf 2)$} 
			
			\put(54,23){\line(-1,-1){9.75}}    
			\put(54,33){\line(-1,-1){9.75}}
			\put(44,33){\line(-1,-1){9.75}}
			\put(44,43){\line(-1,-1){9.75}}
			\put(34,23){\line(1,-1){9.75}}
			\put(34,33){\line(1,-1){9.75}}
			\put(44,33){\line(1,-1){9.75}}
			\put(44,43){\line(1,-1){9.75}}
			\put(44,23){\line(0,-1){9.75}}
			\put(34,33){\line(0,-1){9.75}}
			\put(44,33){\line(0,-1){9.75}}
			\put(54,33){\line(0,-1){9.75}}
			\put(44,43){\line(0,-1){9.75}}
			
			\put(44,5){\makebox(0,0)[cc]{Figure $3$}}  
			
                \end{picture}
\end{center}
\label{fig 4}
\end{figure}

\vspace{.5cm}
The following theorem was proved in \cite[Corollary 5.3]{CoSa23b}.
\begin {theorem} \label{TDisc}
The variety $\mathbb{AG}$ is a discriminator variety.
\end{theorem}

\begin{definition} $\mathbf A \in \mathbb{V}$ has {\bf the congruence extension property} {\rm(CEP)}, if, for every subalgebra $\mathbf B$ of $\mathbf A$, {\rm Con(}$\mathbf B)= \{\theta|_B : \theta \in {\rm Con(}\mathbf A)\}.$  A variety V has {\rm CEP} if all its members have CEP. 
\end{definition}

The following corollary is immediate from Theorem \ref{TDisc} and well-known results from universal algebra (see \cite[Chapter 4]{BuSa81} or \cite{We78}).

\begin{Corollary} \label{CorCEP}
\begin{thlist}
\item[{\rm1}] The algebra $\mathbf 2$ is primal, 

\item[{\rm2}]  $\mathbf{3_{dblst}}$, $\mathbf{3_{dmst}}$, $\mathbf{4_{dmba}}$ are quasiprimal,

\item[3]  All subvarieties of $\mathbb{AG}$ have {\rm CEP}. 
\end{thlist}
\end{Corollary}

\medskip

{\bf Equational Bases for subvarieties of $\mathbb{AG}$ }

\medskip
Equational bases, given below, for the subvarieties of the variety $\mathbb{AG}$ were obtained in \cite{CoSa23b}.  We recall these equational bases below as they will be needed in what follows.  In this subsection, let us abbreviate ``defined, modulo $\mathbb{AG}$, by'' to ``defined by.''

\begin{Theorem} \cite{CoSa23b} \label{theo20230711}
\begin{thlist}
\item
$\{x^* \approx x'\}$ is a base, modulo $\mathbb{AG}$, for
the variety $\mathbb{V}(\mathbf{2}$).

\item
The variety $\mathbb{V}(\mathbf{3_{dblst}})$ is defined by 

the identity: $x \lor x' \approx 1$. 

\item
 $\mathbb{V}(\mathbf{3_{klst}})$ is defined by the identities:
$x{^*}' \approx x^{**}$, and $x'' \approx x$.

\item
  $\mathbb{V}(\mathbf{4_{dmba}})$ is defined by 

 the identity: $x \lor x^*\approx 1$.

\item
 $\mathbb{G}= \mathbb{V}(\{\mathbf{3_{dblst}}, \mathbf{3_{klst}}\})$ is defined by the identity:  $x{^*}' \approx x^{**}$.  

\item   \label{C11}
 $\mathbb{V}(\{\mathbf{3_{dblst}}, \mathbf{4_{dmba}}\})$ is defined by the identity: 

{\rm(J)} $x' \lor y^* \lor z \approx (x' \lor y)^* \lor (x' \lor z)$.  

\item
 $\mathbb{V}(\{\mathbf{3_{klst}}, \mathbf{4_{dmba}}\})$ is defined by 
 the identity: $x'' \approx x$. 
\end{thlist}
\end{Theorem}

\vspace{.5cm}
\section{Amalgamation Property (AP) in the subvarieties of $\mathbb{AG}$} \label{secnum}

In this section we investigate the amalgamation property (AP, for short) in the subvarieties of $\mathbb{AG}$.

The interest in AP goes back to Schreier \cite{Sc27} and Neuman \cite{Ne48, Ne49}.  In fact, AP appeared implicitly even earlier in the Galois theory of field extensions, since it is AP that lets us to consider all the extensions of a given base field as subfields of one large extension. 
 AP appears in a universal algebraic setting in Fra\"iss\'{e} \cite{Fr54}.   For early applications, see \cite{Jo56}.
For the importance of the Amalgamation Property see J\'{o}nsson \cite{Jo65} and Gr\"{a}tzer \cite{Gr71}). 
For a more comprehensive history the reader is referred to \cite{CzPi96} (see also \cite{AlBu88, Ba72a, Ba72b, Be83, Co69, CzPi96, Fr54, GiLeTs15, GrLa71, GrJoLa73, Ja95, Je21, Jo56, Jo65, Jo84, Jo90, KiMaPrTh83, MeMoTs14, Ne48, Ne49,   Pi72a, Pi72b, Pi71, PoTs82, Ya74}).

By a {\bf diagram} in a class $\mathbb{K}$ of algebras we mean a quintuple $\langle \mathbf A, f, \mathbf B, g, \mathbf C\rangle$, where $\mathbf A$, $\mathbf B$, $\mathbf C \in \mathbb K$ and $f : \mathbf A \mapsto \mathbf B$ and  $g : \mathbf A \mapsto \mathbf C$ are embeddings. By an {\bf amalgam} of this diagram in $\mathbb K$ we mean a triple $\langle f_1, g_1, \mathbf D\rangle$ with $\mathbf D \in \mathbb K$ and with embeddings  $f_1 : \mathbf B \mapsto \mathbf D$ and $g_1 : \mathbf C \mapsto \mathbf D$ such that $f_1f= g_1g$. If such an amalgam exists for the diagram $\langle \mathbf A, f, \mathbf B, g, \mathbf C\rangle$, then we say that the diagram is {\bf amalgamable} in $\mathbb K$.  We say that $\mathbb{K}$ has the {\bf Amalgamation Property} (AP) if every diagram in $\mathbb{K}$ is amalgamable. 
Since f and g in a diagram are embeddings, we can, without loss of generality, simply think of a diagram $\langle \mathbf A, f, \mathbf B, g, \mathbf C\rangle$ as a triple 
$\langle \mathbf A, \mathbf B, \mathbf C\rangle$ such that $\mathbf A \leq  \mathbf B \cap \mathbf C$.  Accordingly, we adopt 
this simplified definition of AP 
in the sequel.

We now examine the Amalgamation Property for non-trivial subvarieties of the variety $\mathbb{AG}$.  
For this purpose we need the following theorem from \cite{GrLa71}.  To state the theorem, we need one more definition.

 An algebra $\mathbf{A}$ in a variety $\mathbb{V}$ is {\bf hereditarily subdirectly irreducible} if every
subalgebra of $\mathbf{A}$ is subdirectly irreducible.

\begin{Theorem} {\rm (}\cite{GrLa71}\rm{)} \label{TGr}        
Let $\mathbb{V}$ be a variety of algebras such that

(1) $\mathbb{V}$ has the Congruence Extension Property (CEP), and 

(2) every subdirectly irreducible algebra in $\mathbb{V}$ is hereditarily subdirectly irreducible. \\
Then $\mathbb{V}$ satisfies the Amalgamation Property if and only if whenever $\mathbf{A, B, C}$ are subdirectly irreducible algebras in $\mathbb{V}$ such that $\mathbf{A} \leq \mathbf{B}  \cap \mathbf{C}$, the diagram $\langle \mathbf{A, B, C}\rangle$ is amalgamable in $\mathbb{V}$.
\end{Theorem}

Let $\mathbb{DMBA}$ denote the variety $\mathbb{V}(\mathbf{4_{dmba}})$, whose members will be called ``De Morgan Boolean algebras.''  Observe that it follows from Theorem \ref{Main1} that the varieties $\mathbb{B}$,  
$\mathbb{RDBLS}\rm t$, $\mathbb{RKLS}\rm t$ and $\mathbb{DMBA}$ are subvarieties of $\mathbb{AG}$.  Note also that, since the variety 
$\mathbb{AG}$ is a discriminator variety by \ref{TDisc}, it follows that every subvariety of  $\mathbb{AG}$ has CEP and every algebra in it is hereditarily subdirectly irreducible.  Hence the above theorem is applicable to all the subvarieties of $\mathbb{AG}$.\\

\subsection{The varieties $\mathbb{B}$,  $\mathbb{RDBLS}\rm t$, $\mathbb{RKLS}\rm t$ and $\mathbb{DMBA}$}

\begin{Theorem} \label{Gamal}  
Let $\mathbb{V}$ be a nontrivial subvariety of $\mathbb{AG}.$  Then
\begin{center}
 $\mathbb{V}$ has the Amalgamation Property if $\mathbb{V} \in \{\mathbb{BA}, \mathbb{RDBLS}\rm t, \mathbb{RKLS}\rm t,
\mathbb{DMBA}\}$.
\end{center}
\end{Theorem}

\begin{proof}
  It is well-known that $\mathbb{BA}$ has the Amalgamation Property.  It also trivially follows from Theorem \ref{TGr} since $\mathbf{2}$ is the only subdirectly irreducible (=simple) Boolean algebra.
 Since $\{\mathbf 2$, $\mathbf{3_{dblst}}\}$, $\{\mathbf 2$, $\mathbf{3_{klst}}\}$, $\{\mathbf 2$, $\mathbf{3_{klst}}\}$ and 
$\{\mathbf 2$, $\mathbf{4_{dmba}}\}$ are, respectively, the sets of subdirectly irreducible algebras in $\mathbb{RDBLS}\rm t$, $\mathbb{RKLS}\rm t$ and $\mathbb{DMBA}$, it is also clear from Theorem \ref{TGr} that the varieties 
$\mathbb{RDBLS}\rm t$, $\mathbb{RKLS}\rm t$ and $\mathbb{DMBA}$ have the AP.
\end{proof}

\medskip
\subsection{The variety $\mathbb{G}$}

\

\begin{Theorem} \label{Gamal1}
The variety $\mathbb{G}$ does not have the Amalgamation Property.
\end{Theorem}

\begin{Proof} 
In view of Theorem \ref{TGr}, it suffices to exhibit a diagram of subdirectly irreducible algebras of $\mathbb{G}$ which is not amalgamable in $\mathbb{G}$.  From Theorem \ref{Th1} we know that the algebras $\mathbf{2, 3_{dblst}, 3_{klst}}$ are subdirectly irreducible in $\mathbb{G}$ and $\mathbf 2$ is a subalgebra of both the algebras $\mathbf{3_{dblst}}$ and $\mathbf{3_{klst}}$.\\
So, consider the diagram $\langle \mathbf{2, 3_{dblst}, 3_{klst}}\rangle$ of subdirectly irreducible algebras in $\mathbb{G}$.
We claim that this diagram is not amalgamable in $\mathbb{G}$.  Suppose our claim is false.  Then there exists an algebra $\mathbf A \in \mathbb{G}$ such that both $\mathbf{3_{dblst}}$ and $\mathbf{3_{klst}}$ are subalgebras of $\mathbf A$.  So, there are $a, b \in \mathbf A $ such that $0<a<1$, $a'=1$, $a^*=0$, $0<b<1$, $b'=b$ and $b^*=0$. 
Now, by (R1), we have
$(a \land a'{^*}') \land (b \lor b^*) = a \land a'{^*}' $, 
which, as  $a'=1$ and $b^*=0$,  
leads to
\begin{equation}\label{(E1)}
a  \leq b. 
\end{equation}  
Next, again using (R1),  we get
 $ (b  \land b'{^*}') \land (a \lor a^*) = b \land b'{^*}' $, from which, in view of $a^*=0$, b'=b, and $b^*=0$, 
we have
\begin{equation}\label{(E2)}
b  \leq a.  
\end{equation} 
 From (\ref{(E1)}) and (\ref{(E2)}) it follows that $a = b$, 
which implies                                                                                                                                                                                                             
$b = 1$, as $a'=1$ and $b'=b$, which is a contradiction since $b \neq 1$.  Thus we conclude that the diagram  $\langle \mathbf{2, 3_{dblst}, 3_{klst}}\rangle$ is not amalgamable, proving the theorem.
\end{Proof}

\begin{remark} The join of two discriminator varieties having {\rm AP} need not have {\rm AP} since $\mathbb{G}$ is the join of $\mathbb{RDBLS}\rm t$, $\mathbb{RKLS}\rm t$, each of which has {\rm AP}, while $\mathbb{G}$ does not have {\rm AP}, as shown in the above theorem  This was also, independently, noticed by S. Burris {\rm(}private communication{\rm)}.  
\end{remark}

\medskip
\subsection{The variety $\mathbb{V}(\mathbf{3_{dblst}}, \mathbf{4_{dmba}})$}

\

\medskip

Here we wish to show that $\mathbb{V}(\mathbf{3_{dblst}}, \mathbf{4_{dmba}})$ does not have the AP, the proof of which is developed through the following lemmas, where we assume $\mathbb{V} := \mathbb{V}(\mathbf{3_{dblst}}, \mathbf{4_{dmba}})$, $\mathbf{A} \in \mathbb{V}$, and $x,y, a, b \in \mathbf{A}$ such that $0<a<1$, $a^*=0$ and $a'=1$, $0<b<1$,
$b \lor b^*=1$, $b'=b$, $b{^*}'=b^*.$

\begin{Lemma} \label{LD} 
We have\\
{\rm (a)}  $x' \lor x'^*  = 1.$   \\
 {\rm(b)} $x'^{**} = x'. $  
\end{Lemma}

\begin{proof}
By the defining identity (J) 
of the variety $\mathbb{V}$ (see (f) of Theorem \ref{theo20230711}), we have
$x' \lor (0^* \lor 0)=(x' \lor 0)^* \lor x' \lor 0 $; so,
$x' \lor 1 =x'^* \lor x' $, 
whence,
 $x' \lor x'^*  = 1,$ proving (a). 
In view of (a), we obtain $x' = x' \lor (x'^* \land x'^{**})= (x' \lor x'^*) \land (x' \lor x'^{**})=1 \land (x' \lor x'^{**})= x' \lor x'^{**} = x'^{**}$.
\end{proof}

Recall that $\mathbb{AG} \models {\rm(L1)}$ and hence,$\mathbb{V}(\mathbf{3_{dblst}}, \mathbf{4_{dmba}}) \models {\rm(L1)}$.

\begin{Lemma} \label{LE}
 $x'' \land x'^* = x \land x'^*. $   
\end{Lemma}

\begin{proof}
\begin{align*}
x \land x'^* &= (x \land x'^*)'^*  &\text{by (L1)}\\ 
&= x'^* \land x'{^*}'^*\\
&= x'''^* \land x'''{^*}'^* \\
&= (x'' \land x'''^*)'^* \\  
&= x'' \land x'''^*  &\text{by (L1)}.\\
&= x'' \land x'^*,
\end{align*}
proving the lemma.
\end{proof}

\begin{Lemma} \label{LE1}
$(a \lor x) \lor [(a \lor x) \land (a \lor x)'^*]^*=1.$
\end{Lemma}

\begin{proof}
As $x'' \geq (x'^* \land x'')$, we get from Lemma \ref{LD} (a) that $x'' \lor (x'^* \land x'')^*  \geq x'' \lor x''^*  =1$; whence,
$ x \lor x'' \lor (x'^* \land x'')^* = x \lor 1 $, 
implying
\begin{equation} \label{070823_01}
	x \lor (x'^* \land x'')^* =1,
\end{equation}
as $x'' \leq x.$ 	
From 
(\ref{070823_01}) 
and Lemma \ref{LE}, we get
that  
$x \lor (x \land x'^*)^* = 1,$ from which,
replacing x by $a \lor x$, we get
\begin{equation} \label{A}
	(a \lor x) \lor [(a \lor x) \land (a \lor x)'^*]^*=1,
\end{equation}
proving the lemma.
\end{proof}

\begin{Lemma} \label{C1}
$b^* \geq (b^* \lor a)'$.
\end{Lemma}

\begin{proof}
From Lemma \ref{LD} (a),            
we have $(x \lor y)' \lor (x \lor y)'^* =1$, whence
 $x' \lor (x \lor y)' \lor (x \lor y)'^* =1.$
So, we conclude
 $x' \lor (x \lor y)'^* = 1,$ as $x' \geq (x \lor y)', $
 which (replacing x by $x'$) implies 
$x'' \lor (x' \lor y)'^*=1$. Hence we have  
\begin{equation} \label{eq1}
 x \lor (x' \lor y)'^*=1. 
 \end{equation}
Then, 
\begin{align*}
b^* &=  b^* \lor [(b{^*}' \lor a)'^* \land (b{^*}' \lor a)'^{**} ]\\
&= [b^* \lor (b{^*}' \lor a)'^*] \land [b^*  \lor (b{^*}' \lor a)'^{**}]\\
&= 1 \land [b^* \lor (b{^*}' \lor a)'^{**}] &\text{ by (\ref{eq1})}\\
&= b^* \lor (b{^*}' \lor a)'^{**}\\
&=b^* \lor (b{^*} \lor a)'  &\text{  by Lemma \ref{LD} (b) and \ $b{^*}' = b^*$}.
\end{align*} 
completing the proof.
\end{proof}

Recall that an Almost Gautama algebra $\mathbf A$ has a Stone retract; so, as is well-known, it satisfies the identity: $(x \land y)^* \approx x^* \lor y^*$.

\begin{Lemma} \label{LI}
$a \lor x \lor (a \lor x)' = 1. $  
\end{Lemma}

\begin{proof}
Now, 

$[(a \lor y) \land x]^* =(a \lor y)^* \lor x^* = (a^* \land y^*) \lor x^* = (0 \land y^*) \lor x^* = x^*$.  Hence, we have
\begin{equation} \label{LH}
[(a \lor y) \land x]^* = x^*.
\end{equation}
Next,
\begin{align*}
a \lor x \lor (a \lor x)' &= a \lor x \lor (a \lor x)'^{**} &\text{ by Lemma \ref{LD}}\\
                                 &= (a \lor x)  \lor[(a \lor x) \land (a \lor x)'^*]^*  &\text{ by  (\ref{LH})}\\
                                 &=1 &\text{ by Lemma \ref{LE1}},
\end{align*}
arriving at the desired conclusion.
\end{proof}

\begin{Theorem} \label{3dbl4amal}
The variety $\mathbb{V}(\mathbf{3_{dblst}}, \mathbf{4_{dmba}})$ does not have the Amalgamation Property.
\end{Theorem}

\begin{proof}
Let $\mathbb{V} = \mathbb{V}(\mathbf{3_{dblst}}, \mathbf{4_{dmba}})$, and
consider the diagram $\langle \mathbf{2; 3_{dblst}, 4_{dmba}}\rangle$ in $\mathbb{V}$.
We claim that this diagram is not amalgamable in $\mathbb{V}$.  Suppose our claim is false.  
Then there exists an algebra $\mathbf A \in \mathbb{V}$ such that both $\mathbf{3_{dblst}}$ and $\mathbf{4_{dmba}}$ are subalgebras of $\mathbf A$.  The former implies that there is an $a \in \mathbf A $ such that $0<a<1$, $a^*=0$, $a'=1$,  
and the latter implies that there exists a $b \in \mathbf A $ such that $0<b<1$,  $b \lor b^*=1$, $b'=b$, $b{^*}'=b^*$.    
From Lemma \ref{LI}, we get    
$ a \lor b^* \lor (a \lor b^*)' = 1$,  which, by Lemma \ref{C1}, further simplifies to
\begin{equation} \label{Equa1}
  a \lor b^* = 1.     
\end{equation}
 Also observe that $b \lor (a \lor b)' = b' \lor (b \lor a)' = [b \land (b \lor a)]' = b'=b$; thus,  
 $b \lor (a \lor b)' = b.$  
Hence, it follows from Lemma \ref{LI}  
that $ a \lor b = 1. $  
Then, $a \lor b^* = (a \lor b^*) \land (a  \lor b) = a \lor (b \land b^*)=a$; hence, 
\begin{equation} \label{Equa2}
 a \lor b^* = a.  
\end{equation}
But then, (\ref{Equa1}) and (\ref{Equa2}) yield that $a =1.$  
Thus, we have arrived at a contradiction which proves the claim.  Thus the theorem is proved.
\end{proof}

\subsection{The variety $\mathbb{V}(\mathbf{3_{klst}}, \mathbf{4_{dmba}})$}

\

\medskip
In this subsection we will examine the variety $\mathbb{V}(\mathbf{3_{klst}}, \mathbf{4_{dmba}})$ with respect to (AP).  Again, we need some lemmas.

Throughout this subsection,  $\mathbb{V} := \mathbb{V}(\mathbf{3_{klst}}, \mathbf{4_{dmba}})$, $\mathbf{A} \in \mathbb{V}$. 

\begin{Lemma} \label{Le1}  Let $x,y \in \mathbf{A}$.  Then 
 $x{^*}'^* \land x^* = x^* \land x'.$ 
\end{Lemma}

\begin{proof}
 \begin{align*}
x^*\land  x'  &= x''^* \land  x'  &\text{ as  $x'' = x$}  \\
&= (x''^* \land  x')'^*  &\text{ by (L1)}\\
&=  x''{^*}{'^*} \land x''^*  \\
&= x{^*}{'^*} \land x^*,  
\end{align*}
which provs the lemma.
\end{proof}

\begin{Lemma} \label{Le3}  
  Let $a,b \in \mathbf A$ such that $a^*=0$, $b'=b$ and $b{^*}' =b^*$.  Then
$b^* \land  (b \land  a)' = 0.$  
\end{Lemma}

\begin{proof}
\begin{align*}
 0 &=b^* \land b \\
 &=  b^* \land b'  &\text{since } b'=b\\  
 &= b^* \land b{^*}'^*    &\text{ by Lemma \ref{Le1}} \\  
 &= b^* \land (b{^*} \lor a^*)'^* &\text{since } a^{*}=0\\
 &=b^* \land (b \land a){^*}{'^*} &\text{ since $\mathbf{A}$ has a Stone algebra-reduct} \\
 & = b^* \land (b \land a)'  &\text{ by Lemma \ref{Le1}},           
 \end{align*}
  proving the lemma.  
  \end{proof}

\begin{Theorem} \label{3klst4amal}
The variety $\mathbb{V}(\mathbf{3_{klst}}, \mathbf{4_{dmba}})$ does not have the Amalgamation Property.
\end{Theorem}

\begin{proof}
Let $\mathbb{V} = \mathbb{V}(\mathbf{3_{klst}}, \mathbf{4_{dmba}})$, and
consider the diagram $\langle \mathbf{2; 3_{klst}, 4_{dmba}}\rangle$ in $\mathbb{V}$.
We claim that this diagram is not amalgamable  in $\mathbb{V}$.  Suppose our claim is false.  Then there exists an algebra $\mathbf A \in \mathbb{V}$ such that both $\mathbf{3_{klst}}$ and $\mathbf{4_{dmba}}$ are subalgebras of $\mathbf A$.  So, there are $a, b \in \mathbf A $ such that $0<a<1$, $a^*=0$, $a'=a$,  
and $0<b<1$, $b'=b$, $b{^*}'=b^*$, $b \lor b^*=1$.   
Now,
\begin{align*}
a \land b^* &= (b^* \land b) \lor  (b^* \land a)\\
&= b^* \land  (b \lor a)\\
&= b^* \land  (b' \lor a') &\text{since } b'=b \text{ and } a'=a\\
 &= b^* \land  (b \land a)' \\ 
 &= 0  &\text{by Lemma \ref{Le3}}.  
\end{align*}
Thus, we have
\begin{equation} \label{EqnT}
 a \land b^* = 0.   
\end{equation}
Hence,
\begin{align*}
b^*&= b^* \land 0^*\\
&= b^* \land (b^* \land a)^{*} &\text{by (\ref{EqnT})}\\
&= b^* \land a^*  &\text{since $^*$ is a pseudocomplement}\\
&= 0 &\text{ as $a^*=0$}.          
\end{align*}
Thus, $b^*=0$, which is a contradiction since $b \lor b^*=1$, proving the claim, from which the conclusion follows.  
\end{proof}

\subsection{The variety $\mathbb{AG}$ }

\

\medskip
Here we wish to show that $\mathbb{AG}$ does not have the AP.
Throughout this subsection,  
$\mathbf{A} \in \mathbb{AG}$, and $x, y, a, b \in \mathbf{A}$ such that $a^*=0$,  $a'=1$,
$b'=b$, $b^{**}=b$, and $b{^*}'=b^*$.  Hence, it is also immediate that $b'^{**}=b$.

\begin{Lemma} \label{88}
 $(x \land  x'^*) \lor (x \land  x'^*)^* = 1. $  
\end{Lemma}

\begin{proof}
\begin{align*}
(x \land  x'^*) \lor (x \land  x'^*)^*  &=  (x \land  x'^*)'^* \lor (x \land  x'^*){'{^*}}{^*}  &\text{ by (L1)}\\
&= (x' \lor  x'{^*}')^* \lor (x' \lor  x'{^*}')^{**} \\
&= 1 &\text{ by the Stone identity},\\
\end{align*} 
proving the lemma. 
\end{proof}

\begin{Lemma}\label{231}
  $ [(x \lor a) \land  y]^* = y^*.$  
\end{Lemma}

\begin{proof}  
$[(x \lor a) \land  y]^*= [(x \lor a)^{**} \land y^{**}]^* = [(x^* \land a^*)^* \land y^{**}]^*= (1 \land y^{**})^*=y^*$, as $a^*=0$. 
\end{proof}

\begin{Lemma} \label{232}
  $x \lor a \lor (x \lor a)'^{**} = 1.$   
\end{Lemma}

\begin{proof}
$x \lor (x \land  x'^*)^* \geq (x \land  x'^*) \lor (x \land  x'^*)^* = 1$ by Lemma \ref{88}.  Hence,
\begin{equation} \label{223}
x  \lor (x \land  x'^*)^* = 1.
\end{equation}
Replacing $x$ by $x \lor a$ in (\ref{223}), we get
$(x \lor a) \lor [(x \lor a) \land (x \lor a)'^*]^* =1$, which, by Lemma \ref{231}, simplifies to $(x \lor a) \lor (x \lor a)'^{**} = 1.$ 
\end{proof}

\begin{Lemma} \label{LX}
$y'' \lor y^* = (x' \lor x'{^*})' \lor y'' \lor y{^*}.$
\end{Lemma}

\begin{proof}
From 
(R1) we get
 $y \lor y^* = (x \land x'{^*}') \lor y \lor y^*,$ from which we have 
 $y'' \lor y{^*}'' = (x \land x'{^*}')'' \lor y'' \lor y{^*}''.$  
Hence, 
\begin{align*}
y'' \lor y^* &= y'' \lor y{^*}''  &\text{  by the axiom: $x{^*}'' \approx x^*$}, \\
&= (x \land x'{^*}')'' \lor y'' \lor y{^*}'' \\    
&=  (x' \lor x'{^*}'')' \lor y'' \lor y{^*}'' \\   
&=  (x' \lor x'^*)' \lor y'' \lor y{^*} &\text{  by the axiom: $x{^*}'' \approx x^*$},   
\end{align*}
proving the lemma.
\end{proof}

\begin{Lemma} \label{181}
 $x' \lor x'^* = 1.$    
\end{Lemma}

\begin{proof}
From Lemma \ref{LX} we have
 $a'' \lor a^* = (x' \lor x'^*)' \lor a'' \lor a^*$, which implies $ (x' \lor x'^*)' = 0$, since $a' = 1$ and $a^* = 0.$  It follows that
 $ x''' \lor x'{^*}'' =0'$, whence  
 $ x' \lor x'{^*}'' =1,$  from which, in view of the axiom
 $x{^*}'' = x^*,$  we get
 $x' \lor x'^*=1.$ 
\end{proof}

\begin{Lemma} \label{185}
 $x'^{**} = x'.$  
\end{Lemma}

\begin{proof}
Using Lemma \ref{181} we have 
$x' \lor x'^{**}  
= (x' \lor x'^{**}) \land (x' \lor x'^*) = x' \lor (x'^{**} \land x'^*) =x'$.  Thus $x'^{**} \leq x'$, from which we conclude
 $x'^{**} =x'.$
\end{proof}

\begin{Lemma} \label{214}
 $x' \lor y' = (x'^* \land  y'^*)^*.$  
\end{Lemma}

\begin{proof}
By Lemma \ref{185}, we have $(x'^* \land  y'^*)^* = x'^{**}\lor y'^{**}  
= x' \lor y'.$  
\end{proof}

\begin{Lemma} \label{215}
 $(x \lor y)'' =  (x''^* \land  y''^*)^*. $  
\end{Lemma}

\begin{proof}
Replacing $x$ by $x'$ and $y$ by $y'$  in Lemma \ref{214},  
we get 
$x'' \lor y'' = (x''^* \land  y''^*)^*$, which implies 
$(x \lor y)'' =  (x''^* \land  y''^*)^*. $  
\end{proof}

\begin{Lemma} \label{256}
 $a \lor x \lor x' = 1.$  
\end{Lemma}

\begin{proof}  
 By Lemma \ref{232} and Lemma \ref{185} we have 
 $x \lor a  \lor (x \lor a)' = 1$, which implies $x \lor a \lor x' =1$, since $x' \geq (x \lor a)'$.
\end{proof}

\begin{Theorem} \label{AGamal}
The variety $\mathbb{AG}$ does not have the Amalgamation Property.
\end{Theorem}

\begin{proof}
Consider the diagram $\langle \mathbf{2; 3_{dblst}, 4_{dmba}}\rangle$ in $\mathbb{AG}$.
We claim that this diagram is not amalgamable  in $\mathbb{AG}$.  Suppose our claim is false.  Then there exists an algebra $\mathbf A \in \mathbb{AG}$ such that both $\mathbf{3_{klst}}$ and $\mathbf{4_{dmba}}$ are subalgebras of $\mathbf A$.  So, there are $a, b \in \mathbf A $ such that $0<a<1$, $a^*=0$, $a'=1$,  
 $0<b<1$,  $b \lor b^*=1,$ $b'=b$, and $b{^*}' =b^*$.  Hence it is clear that $b^{**}=b$.
Now, from Lemma \ref{256}, we have  
 $a \lor b \lor b' = 1$ and  $a \lor b^* \lor b{^*}' = 1,$  implying 
$a \lor b = 1$ and $a \lor b^* = 1$.  
Thus, 
$ a=a \lor (b \land  b^*)=(a \lor b) \land (a \lor b^*) = 1 \land 1=1$, whence
$a = 1,$  
which is a contradiction since $a \neq 1$,  
proving the claim. It follows that the variety $\mathbb{AG}$ does not have the Amalgamation Property. 
\end{proof}


We are ready to give our main result of the paper.

\begin{Theorem} \label{Main2}
Let $\mathbb{V}$ be a nontrivial subvariety of $\mathbb{AG}.$  Then
\begin{thlist}
\item[1] $\mathbb{V}$ has the Amalgamation Property if $\mathbb{V} \in \{\mathbb{BA}, \mathbb{RDBLS}\rm t, \mathbb{RKLS}\rm t,
\mathbb{DMBA}\}$.
\item[2] 
  $\mathbb{V}$ does not have the Amalgamation Property if

\quad $\mathbb{V} \in \{\mathbb{G},  
\mathbb{V}(\mathbf{3_{dblst}}, \mathbf{4_{dmba}}),  \mathbb{V}(\mathbf{3_{klst}}, \mathbf{4_{dmba}}), \mathbb{AG} \}$.  
          
\end{thlist}
\end{Theorem}

\begin{proof}
Observe that (1) is just a restatement of Theorem \ref{Gamal}, while (2) follows immediately from Theorems \ref{Gamal1}, \ref{3dbl4amal}, \ref{3klst4amal} and \ref{AGamal}.
\end{proof}

We end this section with the following remark:
\begin{remark}
 The variety $\mathbb{AGH}$ of Almost Gautama Heyting algebras was also introduced in \cite{CoSa23b}, where it was proved that $\mathbb{AGH}$ is term-equivalent to $\mathbb{AG}$ and their lattices of subvarieties are isomorphic.  {\rm(}For the definition of   the variety $\mathbb{AGH}$ we refer the reader to \cite{CoSa23b}.{\rm)}  Hence, as a consequence of Theorem \ref{Main2}, a theorem analogous to Theorem \ref{Main2} for the variety $\mathbb{AGH}$ also holds.  The actual statement of that theorem can be easily written down by the reader with the help of \cite{CoSa23b}.
\end{remark}

\vspace{.6cm}
\section{Applications}

In this section we give several applications of Theorem \ref{Main2} and of proofs of Theorems in Section \ref{secnum}.   We examine the following properties for the subvarieties of $\mathbb{AG}$ : Transferability property (TP), having enough injectives (EI), Embedding Property, Bounded Obstruction Property and having a model companion.

\begin{definition} Let $\mathbb{V}$ be a variety and let $A \in \mathbb{V}$.
 $\mathbb{V}$ has the {\bf transferability property} {\rm (TP)} or {\bf transferable injection property} {\rm (TIP)} if, for every $A, B, C \in \mathbb{V}$  such that $f : A \mapsto B$ an embedding and $g : A \to C$ a homomorphism, there exist D in V, a homomorphism $f_1 : B \to D$ and an embedding $g_1 : C \to D$ such that $f_1 f = g_1 g.$
\end{definition}

The following result appears in Bacsich \cite{Ba72a}. 

\begin{Theorem} \label{TB} Let $\mathbb V$ be a variety.  Then
 $\mathbb V$ has {\rm (TP)} if and only if V has {\rm AP} and {\rm CEP}.
\end{Theorem}

\begin{Corollary} \label{CorTP}
Let $\mathbb{V}$ be a nontrivial subvariety of $\mathbb{AG}.$  Then\\
 $\mathbb{V}$ has {\rm (TP)} if and only if $\mathbb{V} \in \{\mathbb{BA}, \mathbb{RDBLS}\rm t, \mathbb{RKLS}\rm t, \mathbb{DMBA}\}$.
\end{Corollary}

\begin{proof}
Apply Corollary \ref{CorCEP}, Theorem \ref{Main2},  and Theorem \ref{TB}. 
\end{proof}

\begin{definition}
Let $\mathbb V$ be a variety.
\begin{thlist}

\item[1] {\bf $\mathbf A$ is {\bf injective} in $\mathbb V$, if for every $\mathbf A$, $\mathbf B$ in $\mathbb V$, for every embedding $f : B \mapsto C$ and every homomorphism $g: B \to A$, there is a homomorphism $h: C \to A $ such that $hf =g.$}

\item[2]  $\mathbb{V}$ has {\bf enough injectives} {\rm (EI)} if every algebra in $\mathbb{V}$ can be embedded in an injective algebra in $\mathbb{V}$.

\item[3] $\mathbb V$ is {\bf residually small} {\rm (RS)} if there exists a cardinal $\kappa$ such that the size of every    
subdirectly irreducible algebra in $\mathbb V$ is $\leq \kappa$.
\end{thlist}
\end{definition}

The following result is from Banaschewski \cite{Ba70}
\begin{Theorem} \label{TBana} \cite{Ba70}
A variety $ \mathbb{V}$ has {\rm (EI)} if and only if $\mathbb V$ has {\rm (TP)} and is {\rm (RS)}.
\end{Theorem}

\begin{Corollary}
Let $\mathbb{V}$ be a nontrivial subvariety of $\mathbb{AG}.$  
\begin{thlist}
\item[a]  If $\mathbb{V} \in \{\mathbb{V}(\mathbf{2}), \mathbb{V}(\mathbf{3_{dblst}}),             
\mathbb{V}(\mathbf{3_{klst}}), \mathbb{V}(\mathbf{4_{dmba}})\}$, then
$\mathbb{V}$ has EI. 

\item[b]  If $\mathbb{V} \in \{\mathbb{G},  
\mathbb{V}(\mathbf{3_{dblst}}, \mathbf{4_{dmba}}),  \mathbb{V}(\mathbf{3_{klst}}, \mathbf{4_{dmba}}), \mathbb{AG} \}$, then
$\mathbb{V}$ does not have EI. 
\end{thlist}

\end{Corollary}

\begin{proof}
It is clear that every subvariety of $\mathbb{AG}$ is RS.  Then apply Corollary \ref{CorTP} and Theorem \ref{TBana}.  
\end{proof}

\vspace{.5cm}
\subsection{Amalgamation Classes}  
\

\medskip

Recall that a  diagram in a class $\mathbb{K}$ of algebras is a quintuple $\langle A, f, B, g, C\rangle$ with $\mathbf A$, $\mathbf B$, $\mathbf C \in \mathbb K$ and $f : \mathbf A \mapsto \mathbf B$ and  $g : \mathbf A \mapsto \mathbf C$ embeddings, and an amalgam in $\mathbb K$ of this diagram is a triple $\langle f_1, g_1, \mathbf D\rangle$ with $\mathbf D \in \mathbb K$ and with $f_1 : \mathbf B \mapsto \mathbf D$ and $g_1 : C \mapsto D$ embeddings such that $f_1f= g_1g$. If such an amalgam exists, then we say that the diagram is {\it amalgamable} in $\mathbb K$.
An algebra $\mathbf{A}$ is called an {\bf amalgamation base} in a class $\mathbb{K}$ of algebras if every diagram $\langle \mathbf{A}, f, \mathbf{B}, g, \mathbf{C}\rangle$ is amalgamable in $\mathbb K$. 
   Let Amal(K) := $\{\mathbf A : \mathbf A \text{ is an amalgamation base in } \mathbb{K}\}$.  Then Amal(K) is called the {\bf amalgamation class} of $\mathbb{K}$.\\

\begin{Theorem} {\rm(Bergman \cite{Be83})} \label{TBerman} Let $\mathbb{V}$ be a finitely generated discriminator variety. Then {\rm Amal}($\mathbb V$) is finitely axiomatizable (i.e., is definable by a finite set of first-order sentences.)
\end{Theorem}

As a consequence of the above theorem and Theorem \ref{Main2} we obtain the following corollary. 
\begin{Corollary} \label{finax}
If $\mathbb{V} \in \{\mathbb{G},  
\mathbb{V}(\mathbf{3_{dblst}}, \mathbf{4_{dmba}}),  \mathbb{V}(\mathbf{3_{klst}}, \mathbf{4_{dmba}}), \mathbb{AG} \}$, then
 {\rm Amal}($\mathbb{V})$ is finitely axiomatizable. 
\end{Corollary}

The following problems that arise naturally from the preceding corollary are open.\\

{\bf PROBLEM 1}: Can we improve Corollary \ref{finax} further?   In other words,\\
 if $\mathbb{V} \in \{\mathbb{G},  
\mathbb{V}(\mathbf{3_{dblst}}, \mathbf{4_{dmba}}),  \mathbb{V}(\mathbf{3_{klst}}, \mathbf{4_{dmba}}), \mathbb{AG} \}$, 
is {\rm Amal}($\mathbb{V})$ a Horn class, for example?   \\

In this connection, one of the reviewers has remarked that the conclusion of Theorem \ref{TBerman} can be improved to: {\rm Amal}($\mathbb V$) is axiomatizable by a finite set of finite identities.  Therefore, we will state this as a conjecture below.\\

{\bf Conjecture} (due to one of the reviewers): Let $\mathbb{V}$ be a finitely generated discriminator variety. Then {\rm Amal}($\mathbb V$) is axiomatizable by a finite set of identities.\\

{\bf PROBLEM 2}:  Is there an algorithm to find that defining first-order formula, guaranteed by the above corollary for any of the varieties mentioned in the above corollary?\\

{\bf PROBLEM 3}: Find an axiomatization for {\rm Amal}($\mathbb{V})$, where $\mathbb{V}$ is any variety mentioned in Corollary \ref{finax}.
\\

The following theorem is immediate from the proofs of Theorem \ref{Gamal1}, Theorem \ref{3dbl4amal}, Theorem \ref{3klst4amal} and Theorem \ref{AGamal}.

\begin{Theorem} 
If $\mathbb{V} \in \{\mathbb{G},  
\mathbb{V}(\mathbf{3_{dblst}}, \mathbf{4_{dmba}}),  \mathbb{V}(\mathbf{3_{klst}}, \mathbf{4_{dmba}}), \mathbb{AG} \}$, then
$\mathbf 2 \nin Amal(\mathbb{V})$. 
\end{Theorem}

\vspace{.5cm}
\subsection{Embedding Property}

\

\medskip
We will now examine which of the subvarieties of $\mathbb{AG}$ have the (well-known) Embedding Property.

\begin{definition} We say that a variety $\mathbb V$ has the {\bf Embedding Property} if and only if for any two algebras $\mathbf A$ and $\mathbf B$ in $\mathbb V$,
there exists an algebra $\mathbf C$ in $\mathbb V$ into which both $\mathbf A$ and $\mathbf B$ can be embedded.
\end{definition}

\begin{definition} 
An algebra $\mathbf M$ in a variety $\mathbb V$ is said to be {\bf  $\mathbb V$-minimal} if $\mathbf M$ is, up to isomorphism, a subalgebra of every algebra in $\mathbb V$.
\end{definition}
As far as we know, the following lemma is new.

\begin{Lemma} \label{LEmb}
Let a variety $\mathbb V$ possess a $\mathbb V$-minimal algebra $\mathbf M$.  Then $\mathbb V$ has the {\rm (AP)} if and only if  $\mathbb V$ has the Embedding Property. 
\end{Lemma}

\begin{proof} Suppose $\mathbb{V}$ satisfies the hypothesis and has the (AP).  
Let $\mathbf A$ and $\mathbf B$ be two algebras in $\mathbb V$.  
Since $\mathbf M$ is a subalgebra of both $\mathbf A$ and $\mathbf B$, we can consider the diagram $\langle  \mathbf M, \mathbf A, \mathbf B \rangle$.  Then
as $\mathbb V$ has (AP), there exists an algebra $\mathbf C \in \mathbb K$ such that $\mathbf A$ and $\mathbf B$ are subalgebras of $\mathbf C$, implying that $\mathbb{V}$ has the Embedding Property.  The converse is trivial.  
\end{proof}

\begin{Corollary} 
\begin{thlist}
\item[a]  If $\mathbb{V} \in \{\mathbb{V}(\mathbf{2}), \mathbb{V}(\mathbf{3_{dblst}}),             
\mathbb{V}(\mathbf{3_{klst}}), \mathbb{V}(\mathbf{4_{dmba}})\}$, then
$\mathbb{V}$ has the Embedding Property. 

\item[b]  If $\mathbb{V} \in \{\mathbb{G},  
\mathbb{V}(\mathbf{3_{dblst}}, \mathbf{4_{dmba}}),  \mathbb{V}(\mathbf{3_{klst}}, \mathbf{4_{dmba}}), \mathbb{AG} \}$, then
$\mathbb{V}$ fails to have the Embedding Property. 
\end{thlist}
\end{Corollary}

\begin{proof}
Since every subvariety $\mathbb V$ of $\mathbb{AG}$ has the algebra $\mathbf 2$ as $\mathbb V$-minimal, the corollary is immediate from 
Theorem \ref{Main2} and Lemma \ref{LEmb}.
\end{proof}

\smallskip
\subsection{Bounded Obstruction Property}

\

\medskip
 Albert and Burris \cite{AlBu88} have introduced the notion of the bounded obstruction property and shown certain relatioships between the bounded obstructions, model companions, and amalgamation classes.    

\begin{definition} \cite{AlBu88}
\begin{thlist}
\item[1] Let $\mathbb{K}$ be an elementary class, and suppose that the diagram
$\langle \mathbf A, f, \mathbf B, g, \mathbf C\rangle$ has no amalgam in $\mathbb{K}$.  An {\bf obstruction} is any subalgebra $\mathbf C'$ of $\mathbf C$ such that $(\mathbf A', f|_{\mathbf A'}, \mathbf B, g|_{\mathbf A'}, \mathbf{C'})$ has no amalgam in $\mathbb{K}$, where 
$\mathbf{A'} = g^{-1} (\mathbf{C'}).$

\item[2] Let $\mathbb{K}$ be a locally finite elementary class.  $\mathbf{Amal(K)}$ has the {\bf bounded obstruction property (BOP)} with respect to $\mathbb{K}$ if for every $k \in \omega$, there exists an $n \in \omega$ such that the following holds:

If $C \in Amal(\mathbb{K})$, $|\mathbf B| < k$ and the diagram $\langle \mathbf A,f,\mathbf B,g, \mathbf C\rangle$ has no amalgam in $\mathbb{K}$, then there is an obstruction $\mathbf C' \leq \mathbf C$ such that $|\mathbf C'| < n$.
\end{thlist}
\end{definition}

\begin{Theorem} {\rm (Albert and Burris \cite{AlBu88})} \label{AlBu}  Let V be a locally finite variety. Then $Amal(V)$ satisfies the bounded obstruction property if and only if Amal(V) is an elementary class.
\end{Theorem}

The following corollary is immediate from Theorem \ref{AlBu}, Theorem \ref{Main2} and Theorem \ref{finax}.

\begin{Corollary}
Let $\mathbb{V}$ be a nontrivial subvariety of $\mathbb{AG}.$  \\
 If  
\noindent $\mathbb{V} \in \{\mathbb{G},  
\mathbb{V}(\mathbf{3_{dblst}}, \mathbf{4_{dmba}}),  \mathbb{V}(\mathbf{3_{klst}}, \mathbf{4_{dmba}}), \mathbb{AG} \}$, then $Amal(\mathbb V)$ satisfies the bounded obstruction property.           
\end{Corollary}

\medskip
\subsection{Model Companions of the subvarieties of $\mathbb{AG}$}

\

\smallskip
If $\mathbb K$ is a class of (first-order) structures, let $S(\mathbb K)$ denote the class of all substructures of $\mathbb K$.
 If T is a set of (first-order) sentences then $\mathbf{M(T)}$ denotes the class of models of T, i.e. the class of all
structures satisfying all the sentences in T. Two theories T and  T$_1$ are
{\bf mutually model-consistent} if $S(\mathbf{M(T)}) = S(\mathbf{M(T_1)})$, i.e. every model of T can be
embedded in a model of T$_1$ and vice-versa.  T$_1$ is {\bf model-complete} if $\mathbf A, \mathbf B \in \mathbf{M((T_1)}$                                                                                                                         and $\mathbf A \in S(\mathbf B)$ imply $\mathbf A$ is an elementary substructure of $\mathbf B$.  T$_1$ is the {\bf model companion} of T
if (i) T and T$_1$ are mutually model-consistent, and (ii) T$_1$ is model-complete.  A class $\mathbb K$ has a {\bf model companion} if Th($\mathbb K$) has a model companion.

The following result is well-known (see, for example, \cite{Li82}).

\begin{Theorem} {\rm(\cite{Li82})} \label{model}   
If a class $\mathbb K$ is locally finite, 
has a finite language and has the amalgamation property, then $\mathbb K$ has a model companion.
\end{Theorem}

The following corollary is immediate from Theorem \ref{model} and Theorem \ref{Main2}.

\begin{Corollary}
Let $\mathbb{V}$ be a nontrivial subvariety of $\mathbb{AG}.$  

 If $\mathbb{V} \in \{\mathbb{BA}, \mathbb{RDBLS}\rm t, \mathbb{RKLS}\rm t,
\mathbb{DMBA}\}$, then $\mathbb V$ has a model companion.
\end{Corollary}

\medskip
\section{Concluding Remarks}
There is a strong connection between (certain variations of) the amalgamation property of varieties of algebras and the (variations of) the interpolation property of their corresponding logics.   
In the forthcoming paper \cite{CoSa24a}, which is a continuation of the present paper, we will investigate the connection between (certain variations of) the amalgamation property of the subvarieties of the variety of Almost Gautama algebras and the interpolation property of their corresponding logics described in \cite{CoSa23b}.  Finally, we would like to mention that a far-reaching generalization of Almost Gautama algebras, called ``Quasi-Gautama algebras,'' is under investigation in \cite{CoSa24b}. \\

\end{document}